\documentclass[a4paper,12pt,reqno]{amsart}
\usepackage[utf8]{inputenc}
\usepackage{amsmath}
\usepackage{amsthm,amssymb,amsfonts}
\usepackage{graphicx}
\usepackage{hyperref}
\usepackage{url}

\usepackage{geometry}
\usepackage{array,bm}
\usepackage{color}
\usepackage{verbatim} 

\newtheorem{theorem}{Theorem}

\newtheorem{corollary}[theorem]{Corollary}

\newtheorem{remark}[theorem]{Remark}

\newtheorem{lemma}[theorem]{Lemma}

\makeatletter
\@namedef{subjclassname@2020}{\textup{2020} Mathematics Subject Classification}
\makeatother

\begin{document}
\title{On a minimisation problem related to the solenoidal uncertainty}

\author{Yi C. Huang} 
\address{Yunnan Key Laboratory of Modern Analytical Mathematics and Applications, Yunnan Normal University, Kunming 650500, People's Republic of China}
\address{School of Mathematical Sciences, Nanjing Normal University, Nanjing 210023, People's Republic of China}
\email{Yi.Huang.Analysis@gmail.com}
\urladdr{https://orcid.org/0000-0002-1297-7674}
\author{Tohru Ozawa} 
\address{Department of Applied Physics, Waseda University, Tokyo 169-8555, Japan}
\email{txozawa@waseda.jp}
\author{Xinhang Tong}
\address{School of Mathematical Sciences, Nanjing Normal University, Nanjing 210023, People's Republic of China}
\email{letterwoodtxh@gmail.com}

\date{\today} 


\thanks{Research of YCH is partially supported by the National NSF grant of China (no. 11801274), 
the JSPS Invitational Fellowship for Research in Japan (no. S24040),
and the Open Projects from Yunnan Normal University (no. YNNUMA2403) and Soochow University (no. SDGC2418).
Research of TO is partially supported by JSPS Grant-in-Aid for Scientific Research (S) $\sharp$JP24H00024.
YCH would like to thank Nurgissa Yessirkegenov for pointing to Maz'ya's problem on solenoidal uncertainty
and his student Jian-Yang Zhang for helpful assistance to XT in preparing this manuscript.}

\maketitle
\begin{abstract}
We study Hamamoto's expanding square argument towards a 1-D minimisation problem related to the sharp solenoidal uncertainty principle.
Working in the right function space, we recast the involved interpolation type inequality into an exact equality,
where the vanishing of the remainder term characterises the extremisers via the confluent hypergeometric functions. 
In the process we also remove some unnecessary constraints on the prescribed parameters.
\end{abstract}

\section{Introduction}
In \cite{Mazya2018}, V. Maz'ya proposed an \textbf{Open Problem} for the sharp dimensional constant in the following Heisenberg-Pauli-Weyl uncertainty principle inequality 
$$
\int_{\mathbb{R}^N} |\nabla \mathbf{u}|^2 \, dx \, \int_{\mathbb{R}^N} |\mathbf{u}|^2 |x|^2 \, dx \geq C_N \left( \int_{\mathbb{R}^N} |\mathbf{u}|^2 \, dx \right)^2
$$
for solenoidal (namely, divergence-free) fields $\mathbf{u}: \mathbb{R}^N\rightarrow \mathbb{R}^N$. 
It turns out 
$$C_N=\begin{cases}
\frac14(N+2)^2, \quad &N\in\{1,2\},\\
\frac14 \left(\sqrt{(N-2)^2+8}+2\right)^2, \quad &N\geq3,
\end{cases}$$
see Hamamoto \cite{HAMAMOTO2023202} for $N\geq3$ and Cazacu-Flynn-Lam \cite{CFL22} for $N\in\{1,2\}$. 
In particular, by poloidal-toroidal and spherical-harmonics decompositions, Hamamoto reduced the question to the following elegant 1-D minimisation problem:
\begin{equation}\label{eq1}
 \inf\limits_{f \not\equiv 0} \frac{\left(\int_{0}^{+\infty} (f'')^2 x^{\mu+1} dx\right) \left( \int_{0}^{+\infty} \left( x^2 (f')^2 - \varepsilon f^2 \right) x^{\mu-1} dx\right)}{\left( \int_{0}^{+\infty} (f')^2 x^{\mu} dx \right)^2}.
\end{equation}
Here $\varepsilon$ and $\mu$ are reasonable constants coming from dimensional reduction.

Let us now introduce the natural weighted Sobolev space in which we shall solve the minimisation problem \eqref{eq1}
(clearly, one should also assume $f'\not\equiv 0$). 
For a twice differentiable real-valued function $f$ on the half-line \( \mathbb{R}_+=(0,+\infty) \) and \(\mu\in\mathbb{R}\), set
$$
\|f\| := \left( \int_{0}^{+\infty} \left(x^{\mu+1}(f'')^2 +(x^{\mu+1}+x^{\mu-1})(f')^2+x^{\mu-1}f^{2}\right) dx \right)^{1/2}.
$$
Let \(\bar{\mathcal{E}}_\mu(\mathbb{R}_+)\) be the Hilbert space completion with respect to $\|\cdot\|$  of  \[ \mathcal{E}_\mu(\mathbb{R}_+):= \{ f \in C^{\infty}(\mathbb{R}_+): \, \|f\| < +\infty \}.\] 
By working in \(\bar{\mathcal{E}}_\mu(\mathbb{R}_+)\), we see immediately that \eqref{eq1} is \textit{de facto} an interpolation problem\footnotemark
\footnotetext{For a related problem studied using similar ideas, we refer to Hayashi-Ozawa \cite{HayOza17}.
We also refer to \cite{OZAWA2017998} for Schr\"odinger-Robertson uncertainty relations in an algebraic framework.}. 
In \cite{Hamamoto2024} Hamamoto used instead the larger function space induced by 
$$
\|f\|' := \left( \int_{0}^{+\infty} \left(x^{\mu+1}(f'')^2 +(x^{\mu+1}+x^{\mu})(f')^2+x^{\mu-1}f^{2}\right) dx \right)^{1/2},
$$
see \cite[Line 3, Page 655]{Hamamoto2024}.
However, a careful analysis shows that his expanding square argument\footnotemark
\footnotetext{Such a well-known technique was also used for example in \cite{COSTA2008311} and \cite{cazacu2023caffarelli}.} has a flaw since it relies essentially on the cancellation of the term $\int_{0}^{+\infty}(f')^2x^{\mu-1}\,dx$ (see \cite[$I_0{[f]}$, Page 657 and (2.2), Page 658]{Hamamoto2024}), which may diverge as far as the argument is based on the larger function space induced by $\|\cdot\|'$.

We need the following properties of the space $\bar{\mathcal{E}}_\mu(\mathbb{R}_+)$.

\begin{lemma}\label{lemma1}
For every \(f\in \bar{\mathcal{E}}_\mu(\mathbb{R}_+)\), it holds that 
$$
\lim_{x \rightarrow \{0,+\infty\}  } f'(x)x^{\frac{\mu+1}{2}}=\lim_{x \rightarrow \{0,+\infty\}  } f'(x)x^{\frac{\mu}{2}}= \lim_{x \rightarrow \{0,+\infty\} } f(x)x^{\frac{\mu}{2}}=0.
$$
In particular, we also have
$$\lim_{x \rightarrow \{0,+\infty\}} f'(x)f(x)x^{\mu}=0.$$
\end{lemma}

Our main results are stated as follows.

\begin{theorem}\label{theorem2}
Suppose $\mu$ and $\varepsilon$ are two reals with \( \varepsilon \le \frac{\mu^{2}}{4}\). Let \(b=\frac{\mu - \sqrt{\mu^2 - 4\varepsilon}}{2}\).
Then 
\begin{equation}\label{eq0}
\begin{aligned}
&\int_0^{+\infty} \big(x f'' +  x f' + \mu f' + (\mu-b) f\big)^2 x^{\mu-1} dx\\&\qquad= \int_{0}^{+\infty}\left(x^{2}(f')^{2}-\varepsilon f^{2}\right)x^{\mu-1}dx\\
&\qquad\qquad-\left(\sqrt{\mu^{2}-4\varepsilon}+1\right) \int_{0}^{+\infty}(f')^{2}x^{\mu}dx+\int_{0}^{+\infty}(f'')^{2}x^{\mu+1}dx
\end{aligned}
\end{equation}
and 
\begin{equation}\label{eq2}
\begin{aligned}
	\left(\int_{0}^{+\infty} (f'')^2 x^{\mu+1} dx\right)
	\left( \int_{0}^{+\infty} \left( x^2 (f')^2 - \varepsilon f^2 \right) x^{\mu-1} dx\right) \\ \ge \frac{1}{4}\left(\sqrt{\mu^{2}-4\varepsilon}+1\right)^{2}\left( \int_{0}^{+\infty} (f')^2 x^{\mu} dx \right)^2
\end{aligned}
\end{equation}
hold for all \(f\in\bar{\mathcal{E}}_\mu(\mathbb{R}_+) \).
\end{theorem}

Note that \eqref{eq2} was already established in \cite[Theorem 5.2]{HAMAMOTO2023202} by Laguerre polynomial expansion and then in \cite[Theorem 1]{Hamamoto2024}, though with a flaw as indicated above.
We shall prove \eqref{eq2} in the framework of equalities in which \eqref{eq0} is a special case.
These equalities allow us to characterise the extremisers in a clear-cut way.

\begin{corollary}\label{corollary3}
	Suppose $\mu\neq0$ and $\varepsilon$ are two reals with \( \varepsilon \le \frac{\mu^{2}}{4}\). Let \(b=\frac{\mu - \sqrt{\mu^2 - 4\varepsilon}}{2}\).
The equality in \eqref{eq2} holds if and only if for some \(C\in\mathbb{R} \) and \(\lambda>0\), 
		$$f(x) = Ce^{-\lambda x} \, {}_1F_1(b, \mu, \lambda x)$$
	in the case \(\mu>0\),  or 
		$$f(x) = C(\lambda x)^{1-\mu}e^{-\lambda x} \, {}_1F_1(b+1-\mu, 2-\mu, \lambda x)$$
	in the case \(\mu < 0\). Here, 
	$${}_1F_1(b, \mu, x):=1+\sum_{k=1}^{+\infty}\frac{b(b+1)\cdots (b+k-1)}{\mu(\mu+1)\cdots (\mu+k-1)}\frac{x^k}{k!}$$
	denotes Kummer's confluent hypergeometric function that solves Kummer's ODE
	$$
	xw''+(\mu-x)w'-bw=0.
	$$
\end{corollary}

\section{Proofs of the main results}
\begin{proof}[Proof of Lemma \ref{lemma1}]
We start with the tedious arguments for completeness. Set \[G_{1}(x):=(f'(x))^{2}x^{\mu+1}, \quad G_{2}(x):=(f'(x))^{2}x^{\mu}, \quad G_{3}(x):=(f(x))^{2}x^{\mu}.\] 
	Simple reasoning shows that it suffices to check
	\begin{equation}\label{limits}
		\lim_{x \rightarrow +\infty } G_{1}(x)= \lim_{x \rightarrow +\infty } G_{3}(x)= \lim_{x \rightarrow 0 } G_{2}(x)=\lim_{x \rightarrow 0 } G_{3}(x)=0.
	\end{equation}
	
	(i) By using the elementary inequality \(2ab\le a^{2}+b^{2}\), we get
	\begin{equation}
		\begin{aligned}
			&\int_{1}^{+\infty}\left| \frac{dG_{1}(x)}{dx}\right| dx\\
			&\qquad=\int_{1}^{+\infty}\left| (\mu+1)x^{\mu}(f')^{2}+2x^{\mu+1}f''f'\right|dx\\
			&\qquad\le \int_{1}^{+\infty}\left|(\mu+1)x^{\mu}(f')^{2}+x^{\mu+1}(f'')^{2}+x^{\mu+1}(f')^{2}\right|dx\\
			&\qquad\le \int_{1}^{+\infty} \left(\left|\frac{\mu+1}{2} \right|\left(x^{\mu+1}(f')^{2}+x^{\mu-1}(f')^{2}\right)+x^{\mu+1}(f'')^{2}+x^{\mu+1}(f')^{2}\right)dx\\
			&\qquad\lesssim \left\| f \right\|^{2}<+\infty.
			\nonumber
		\end{aligned}
	\end{equation}
	This implies that
	$$
	G_{1}(+\infty)=G_{1}(1)+\int_{1}^{+\infty}\frac{dG_{1}(x)}{dx}dx
	$$
	is convergent and hence a nonnegative constant. If  \(G_{1}(+\infty)>0\), then \[+\infty=\int_{1}^{+\infty}G_{1}(x)dx\le \left\| f\right\|^{2}<+\infty.\] 
	From this contradiction we deduce \(G_{1}(+\infty)=0\).
	
	(ii) To verify the second limit in \eqref{limits}, we can use similar arguments:  
	\begin{equation}
		\begin{aligned}
			\int_{1}^{+\infty}\left| \frac{dG_{3}(x)}{dx}\right| dx&=\int_{1}^{+\infty}\left| \mu x^{\mu-1}f^{2}+2x^{\mu}f'f\right|dx\\&\le \int_{1}^{+\infty}\left(\left|\mu \right|x^{\mu-1}f^{2}+x^{\mu+1}(f')^{2}+x^{\mu-1}f^{2}\right)dx\\&\lesssim \left\| f \right\|^{2}<+\infty
			\nonumber
		\end{aligned}
	\end{equation}
and hence that	$$
	G_{3}(+\infty)=G_{3}(1)+\int_{1}^{+\infty}\frac{dG_{3}(x)}{dx}dx
	$$
	is convergent and hence a nonnegative constant. If \(G_{3}(+\infty)>0\), then 
	$$
	+\infty=\int_{1}^{+\infty}\frac{G_{3}(x)}{x}dx=\int_{1}^{+\infty}x^{\mu-1}f^{2}dx\le \left\| f \right\|^{2}<+\infty,
	$$
	which is a contradiction. Therefore, we deduce \(G_{3}(+\infty)=0\).
	
	(iii) Now we verify the limits in \eqref{limits} when $x\rightarrow0$. Note as in (ii) that
	\begin{equation}
		\begin{aligned}
			\int_{0}^{1}\left| \frac{dG_{3}(x)}{dx}\right| dx&=\int_{0}^{1}\left| \mu x^{\mu-1}f^{2}+2x^{\mu}ff'\right|dx\\&\le \int_{0}^{1}\left(\left|\mu\right| x^{\mu-1}f^{2}+x^{\mu+1}(f')^{2}+x^{\mu-1}f^{2}\right)dx\\&\lesssim \left\| f \right\|^{2}<+\infty
			\nonumber
		\end{aligned}
	\end{equation}
	and hence that
	$$
	G_{3}(0)=G_{3}(1)-\int_{0}^{1}\frac{dG_{3}(x)}{dx}dx
	$$
	is a nonnegative constant. If \(G_{3}(0)>0\), then as in (ii) we have \[+\infty= \int_{0}^{1}\frac{G_{3}(x)}{x}dx=\int_{0}^{1}x^{\mu-1}f^{2}dx\le \left\| f \right\|^{2}<+\infty,\] which is a contradiction. Thus we deduce \(G_{3}(0)=0\). Similarly,
	\begin{equation}
		\begin{aligned}
			\int_{0}^{1}\left| \frac{dG_{2}(x)}{dx}\right| dx&=\int_{0}^{1}\left| \mu x^{\mu-1}(f')^{2}+2x^{\mu}f''f'\right|dx\\&\le \int_{0}^{1}\left(\left|\mu \right|x^{\mu-1}(f')^{2}+x^{\mu+1}(f'')^{2}+x^{\mu-1}(f')^{2}\right)dx\\&\lesssim \left\| f \right\|^{2}<+\infty
			\nonumber
		\end{aligned}
	\end{equation}
	and hence that 
	$$
	G_{2}(0)=G_{2}(1)-\int_{0}^{1}\frac{dG_{2}(x)}{dx}dx
	$$
	is a nonnegative constant. If \(G_{2}(0)>0\), then \[+\infty= \int_{0}^{1}\frac{G_{2}(x)}{x}dx=\int_{0}^{1}x^{\mu-1}(f')^{2}dx\le \left\| f \right\|^{2}<+\infty,\] which is a contradiction. Thus we deduce \(G_{2}(0)=0\).
	
	Collecting all the justifications above then proves the lemma.
\end{proof}

Now, using the expanding square argument, we prove Theorem \ref{theorem2}.

\begin{proof}[Proof of Theorem \ref{theorem2}]
For every \(f\in\bar{\mathcal{E}}_\mu(\mathbb{R}_+) \), and any $\alpha, \beta, \gamma\in\mathbb{R}$, we have
\begin{equation}\label{eq6}
\begin{aligned}
\int_0^{+\infty} &\big(x f'' + \alpha x f' + \beta f' + \gamma f\big)^2 x^{\mu-1} dx 
\\&= \int_0^{+\infty} \big(f''\big)^2 x^{\mu+1} dx + \alpha^2 \int_0^{+\infty} \big(x f'\big)^2 x^{\mu-1} dx + \beta^2 \int_0^{+\infty} \big(f'\big)^2 x^{\mu-1}dx 
\\&\qquad+\gamma^2 \int_0^{+\infty} f^2 x^{\mu-1} dx +2\alpha\beta \int_0^{+\infty} \big(f'\big)^2 x^\mu dx\\
&\qquad+2\alpha \int_0^{+\infty} f'' f' x^{\mu+1} dx + 2\beta \int_0^{+\infty} f'' f' x^{\mu} dx + 2\gamma \int_0^{+\infty} f''f x^\mu dx
\\& \qquad+2\alpha\gamma \int_0^{+\infty} f' f x^\mu dx +2\beta \gamma \int_0^{+\infty} f' f x^{\mu-1} dx.
\end{aligned}
\end{equation}
Using Lemma \ref{lemma1}
and integration by parts for the first four cross-terms, we get\footnotemark\footnotetext{Here, the involved ``integration by parts" and ``differentiation" are valid even if $\mu\in\{-1, 0\}$.}  
\begin{align*}
2\alpha \int_0^{+\infty} f'' f' x^{\mu+1} dx&=\alpha\left[(f')^{2}x^{\mu+1}\right]_{0}^{+\infty}-\alpha(\mu+1)\int_{0}^{+\infty}(f')^{2}x^{\mu}dx\\&=-\alpha(\mu+1)\int_{0}^{+\infty}(f')^{2}x^{\mu}dx,
\end{align*}
\begin{align*}
2\beta \int_0^{+\infty} f'' f' x^{\mu} dx&=\beta\left[(f')^{2}x^{\mu}\right]_{0}^{+\infty}-\beta\mu\int_{0}^{+\infty}(f')^{2}x^{\mu-1}dx\\&=-\beta\mu\int_{0}^{+\infty}(f')^{2}x^{\mu-1}dx,
\end{align*}
\begin{align*}
2\gamma \int_0^{+\infty} f'' f x^{\mu} dx&=2\gamma\left[ f'fx^{\mu}\right]_{0}^{+\infty}-2\gamma\int_{0}^{+\infty}(f')^{2}x^{\mu}dx-2\gamma\mu\int_{0}^{+\infty}f'fx^{\mu-1}dx\\&=-2\gamma\int_{0}^{+\infty}(f')^{2}x^{\mu}dx-2\gamma\mu\int_{0}^{+\infty}f'fx^{\mu-1}dx,
\end{align*}
and
\begin{align*}
2\alpha\gamma \int_0^{+\infty} f' f x^\mu dx&=\alpha\gamma\left[ f^{2}x^{\mu}\right]_{0}^{+\infty}-\alpha\gamma\mu\int_{0}^{+\infty}f^{2}x^{\mu-1}dx\\&=-\alpha\gamma\mu\int_{0}^{+\infty}f^{2}x^{\mu-1}dx.
\end{align*}
Then \eqref{eq6} can be rewritten as 
\begin{align*}
\int_0^{+\infty} &\big(x f'' + \alpha x f' + \beta f' + \gamma f\big)^2 x^{\mu-1} dx\\
&=\int_0^{+\infty} \big(f''\big)^2 x^{\mu+1} dx + \alpha^2 \int_0^{+\infty} \big(x f'\big)^2 x^{\mu-1} dx + \beta^2 \int_0^{+\infty} \big(f'\big)^2 x^{\mu-1}dx 
\\&\qquad+\gamma^2 \int_0^{+\infty} f^2 x^{\mu-1} dx+2\alpha\beta \int_0^{+\infty} \big(f'\big)^2 x^\mu dx\\
&\qquad-\alpha(\mu+1)\int_{0}^{+\infty}(f')^{2}x^{\mu}dx-\beta\mu\int_{0}^{+\infty}(f')^{2}x^{\mu-1}dx -2\gamma\int_{0}^{+\infty}(f')^{2}x^{\mu}dx\\
&\qquad-2\gamma\mu\int_{0}^{+\infty}f'fx^{\mu-1}dx-\alpha\gamma\mu\int_{0}^{+\infty}f^{2}x^{\mu-1}dx+2\beta \gamma \int_0^{+\infty} f' f x^{\mu-1} dx.
\end{align*}
To cancel out the only remaining cross-term, we choose $\beta=\mu$.
Furthermore, we shall also specify the choice $\gamma=(\mu-b)\alpha$, where $b=\frac{\mu - \sqrt{\mu^2 - 4\varepsilon}}{2}$. 
This leads to
\begin{equation}\label{eq8}
\begin{aligned}
&\int_0^{+\infty} \big(x f'' + \alpha x f' + \beta f' + \gamma f\big)^2 x^{\mu-1} dx\\&\qquad=\alpha^{2}\int_{0}^{+\infty}\left(x^{2}(f')^{2}-b(\mu-b)f^{2}\right)x^{\mu-1}dx\\
&\qquad\qquad-\alpha(\mu+1-2b)\int_{0}^{+\infty}(f')^{2}x^{\mu}dx+\int_{0}^{+\infty}(f'')^{2}x^{\mu+1}dx\\
&\qquad=\alpha^{2}\int_{0}^{+\infty}\left(x^{2}(f')^{2}-\varepsilon f^{2}\right)x^{\mu-1}dx\\
&\qquad\qquad-\alpha\left(\sqrt{\mu^{2}-4\varepsilon}+1\right)\int_{0}^{+\infty}(f')^{2}x^{\mu}dx+\int_{0}^{+\infty}(f'')^{2}x^{\mu+1}dx=:g(\alpha).
\end{aligned}
\end{equation}
From the non-negativity of the above quadratic function $g: \alpha\mapsto g(\alpha)$, we deduce that
$$
\begin{aligned}
&4\left(\int_{0}^{+\infty}\left(x^{2}(f')^{2}-\varepsilon f^{2}\right)x^{\mu-1}dx\right)\left(\int_{0}^{+\infty}(f'')^{2}x^{\mu+1}dx\right)\\
&\qquad\qquad\ge\left(\sqrt{\mu^{2}-4\varepsilon}+1\right)^{2}\int_{0}^{+\infty}(f')^{2}x^{\mu}dx.
\end{aligned}
$$
This together with the equality in \eqref{eq8} for $\alpha=1$ proves the theorem.
\end{proof}

\begin{remark}
Note that \eqref{eq8} for $b=\frac{\mu + \sqrt{\mu^2 - 4\varepsilon}}{2}$ leads to a weaker version of \eqref{eq2}.
\end{remark}

\begin{proof}[Proof of Corollary \ref{corollary3}]	
	The equality in \eqref{eq2} is attained if and only if the quadratic function $g$ in \eqref{eq8} has only one real root, denoted by $\lambda$, namely,
	$$
	\lambda=\frac{\sqrt{\mu^{2}-4\varepsilon}+1}{2}\frac{\int_{0}^{+\infty}(f')^{2}x^{\mu}dx}{\int_{0}^{+\infty}\left(x^{2}(f')^{2}-\varepsilon f^{2}\right)x^{\mu-1}dx}.
	$$
	Note that by the weighted Hardy inequality, $\lambda\geq0$, and we can assume $\lambda>0$ otherwise $f$ is constant, which is not in $\bar{\mathcal{E}}_\mu(\mathbb{R}_+)$ and is excluded from the optimisation problem.
	Suppose $f\in \bar{\mathcal{E}}_\mu(\mathbb{R}_+)$ satisfies the equality in \eqref{eq2}, then we see from \eqref{eq8} that 
	\begin{equation}\label{eq9}
	xf'' + \lambda xf' +\mu f'+(\mu-b)\lambda f=0.
	\end{equation}

Now, to solve \eqref{eq9} we introduce
	$$
	w_\lambda( x):=w(\lambda x):=e^{\lambda x}f(x).
	$$
Then the ODE \eqref{eq9} is transformed into
\begin{align*}
0&=x(e^{-\lambda x}w_\lambda( x))''+(\lambda x+\mu) (e^{-\lambda x}w_\lambda( x))'+(\mu-b)\lambda e^{-\lambda x}w_\lambda( x)\\
&=x(\lambda^{2}e^{-\lambda x}w(\lambda x)-2\lambda^{2}e^{-\lambda x}w'(\lambda x)+\lambda^{2}e^{-\lambda x}w''(\lambda x))\\
&\quad+(\lambda x+\mu)(-\lambda e^{-\lambda x}w(\lambda x)+\lambda e^{-\lambda x}w'(\lambda x))+(\mu-b)\lambda e^{-\lambda x}w(\lambda x)\\
&=\lambda e^{-\lambda x}\left(\lambda xw''(\lambda x)+(\mu-\lambda x)w'(\lambda x)- bw(\lambda x)\right).
\end{align*}
Letting $t=\lambda x$, we get Kummer's (second order) ODE
$$
tw''(t)+(\mu-t)w'(t)-bw(t)=0,
$$
which is known to have two independent solutions (see for example Viola \cite{MR3586206}):
	$$
	\varphi(t) = {}_1F_1(b, \mu, t) \quad \text{and} \quad \psi(t) = t^{1-\mu} {}_1F_1(b+1-\mu, 2-\mu, t).
	$$
        Recall that ${}_1F_1(b, \mu, z)$ is entire in $b$ and $z$, and meromorphic in $\mu$ with poles $\mu=0, -1, -2, \cdots$.
        Noticing that when $\mu>0$, the function \(e^{-t}\varphi: t\mapsto e^{-t}\varphi(t)\) belongs to \( \bar{\mathcal{E}}_\mu(\mathbb{R}_+)\). However, 
	$$
	\int_{0}^{+\infty}\left(e^{-t}\psi''(t)\right)^{2}t^{\mu+1}dt=+\infty,
	$$ 
	since $(\psi''(t))^{2}=O(t^{-2-2\mu})$ as $t\rightarrow 0$.
         In the case $\mu<0$,
	$$
	\int_{0}^{+\infty}\left(e^{-t}\varphi(t)\right)^{2}t^{\mu-1}dt=+\infty,
	$$
	thus $e^{-t}\varphi \notin \bar{\mathcal{E}}_\mu(\mathbb{R}_+)$. Noticing that  
	$$(\psi''(t))^{2}t^{\mu+1}=O(t^{-1-\mu}), \quad (\psi'(t))^{2}t^{\mu-1}=O(t^{-1-\mu}), \quad (\psi(t))^{2}t^{\mu-1}=O(t^{1-\mu})$$ as $t\rightarrow 0$, we can deduce that 
	when $\mu<0$, $e^{-t}\psi\in \bar{\mathcal{E}}_\mu(\mathbb{R}_+)$.
	
	Tracking the above discussions back to $x$ variable, we proved the corollary.
	\end{proof}

\bigskip

\section*{\textbf{Compliance with ethical standards}}

\bigskip

\textbf{Conflict of interest} The authors have no known competing financial interests
or personal relationships that could have appeared to influence this reported work.

\bigskip

\textbf{Availability of data and material} Not applicable.

\bigskip

\bibliographystyle{alpha}
\bibliography{HuaY-OzaT-TonX-RadSUP}

\end{document}